\documentclass[10pt,a4paper,oneside]{article}
\date{}
\usepackage[ansinew]{inputenc}
\usepackage{array}
\usepackage{color}
\usepackage{amsmath}
\usepackage{amsxtra}
\usepackage{amstext}
\usepackage{amssymb}
\usepackage{graphicx}
\usepackage{cite}
\usepackage{latexsym}
\usepackage{amsmath,amsfonts,amssymb,amsthm,epsfig,epstopdf,titling,url,array}
\usepackage[a4paper, total={6.3in, 9.4in}]{geometry}
\usepackage{titlesec} 
\makeatletter
\numberwithin{equation}{section}

\newcommand*{\rom}[1]{\expandafter\@slowromancap\romannumeral #1@}
\newtheorem{cor}{Corollary}[section]
\theoremstyle{definition}
\newtheorem{defn}{Definition}[section]

\newtheorem{rem}{Remark}[section]

\newtheorem{thm}{Theorem}[section]
\newtheorem{lem}{Lemma}[section]

\newcommand{\btheorem}{\begin{theorem}}
	\newcommand{\etheorem}{\end{theorem}}
\begin{document}
		\begin{center}
			\large{\textbf{Fekete-Szeg{\"o} inequality for  Classes of $(p,q)$–Starlike and $(p,q)$–Convex Functions } }

	\vspace{8mm}
	
			Nusrat Raza  \\
	Mathematics Section, Women's College, AMU, Aligarh-202002,\\[0pt]
	nraza.maths@gmail.com\\[0pt]

	and\\
	\vspace{2mm} 
	
		Eman S. A. AbuJarad\\
	Department of Mathematics, AMU, Aligarh-202002,\\[0pt]
	emanjarad2@gmail.com\\[0pt]
	
and \\

{\large Gautam Srivastava\\[0pt]
	Department of Mathematics and Computer Science, Brandon University, 270 18th
	Street, Brandon, Canada, R7A 6A9\\[0pt]
	srivastavag@brandonu.ca\\[0pt]
}

{\large Research Center for Interneural Computing, China Medical University,
	Taichung 40402, Taiwan, Republic of China }

{\large and \\[0pt]
}

{\large H. M. Srivastava \\[0pt]
	Department of Mathematics and Statistics, University of Victoria, Victoria,
	British Columbia V8W 3R4, Canada,\\[0pt]
	harimsri@math.unvi.ca\\[0pt]
}

{\large Department of Medical Research, China Medical University Hospital,
	China Medical University, Taichung 40402, Taiwan, Republic of China }

{\large and \\[0pt]
}
{\large Mohammed H AbuJarad \\[0pt]
	Department of Statistics and Operations Research, AMU, Aligarh-202002,\\[0pt]
	m.jarad@gu.edu.ps\\[0pt]
}

		\vspace{2mm}
		
	\end{center}
	\vspace{5mm}
	{\bf{Abstract:}}
	  In the present paper, the new generalized classes of  $(p,q)$-starlike and $(p,q)$-convex functions are introduced by using the $(p,q)$-derivative operator. Also, the $(p,q)$-Bernardi integral operator for analytic function is defined in the open unit disc $\mathbb{U}=\left\lbrace z\in \mathbb{C}:|z|<1\right\rbrace $ . Our aim for these classes is to investigate the Fekete-Szeg{\"o} inequalities. Moreover, Some special cases of the established results  are  discussed. Further, certain applications of the main results are obtained by applying the  $(p,q)$-Bernardi integral operator .

 	\section{Introduction}
 	
The $q$-analysis is a generalization of the ordinary analysis with out using the limit notation. The first application and usage of the $q$-calculus was  introduced by Jackson~\cite{Jack} and \cite{ackso}. Moreover, several applications in various fields of Mathematics and physics (see for details \cite{sofonea2008some},\cite{srivastava2011some}). Recently, there is an extension of $q$-calculus, denoted by $(p,q)$-calculus which obtained by substituting $q$ by $q/p$ in $q$-calculus. The $(p,q)$-integer was considered by Chakrabarti and Jagannathan \cite{chakrabarti1991p}, see also, \cite{Ber},\cite{bukweli2013quantum} and \cite{sadjang2013fundamental} . The two important geometric properties of analytic functions are starlikeness and convexity. So that, there are many publications in Geometric Function Theory by using the $q$-differential operator, for example, A generalization of starlike functions $\mathcal{S}^{*}$ were investigated by Ismail et al. \cite{Ismail} . Further, Close-to-convexity of a certain family of $q$-Mittag-Leffler functions were studied by \cite{Bansa}. Also, the coefficient inequality $q$-starlike functions were discussed by \cite{uccar2016coefficient} . Recently, Coefficient estimates of $q$-starlike and $q$-convex functions were studied by \cite{seoudy2016coefficient}. Further, new subclasses of analytic functions associated with $q$-differential operators were introduced and discussed, see for example \cite{Aldweby}, \cite{frasin2017new}, \cite{Seoudy}, \cite{Mahmo}\cite{Sriv},\cite{Srivastava} and \cite{uccar2016coefficient}. Motivated by an emerging idea of $(p,q)$-analysis as a generalization of $q$-analysis, in this paper, we extend the idea of $q$-starlikeness and $q$-convexity to $(p,q)$-starlikeness and $(p,q)$-convexity, then we will obtain the  Fekete-Szeg{\"o} inequalities for these classes. Also, we will apply these results on the introduced $(p,q)$-Bernardi integral operator.  \\

 We recall some basic notations and definitions from $(p,q)$-calculus, which are used in this paper .\\
 	
 	 The $(p,q)$-derivative of the function $f$ is defined as\cite{acar2016kantorovich}:
 	\begin{equation}\label{2}
 	D_{p,q}f(z)=\dfrac{f(pz)-f(qz)}{(p-q)z} \hspace{15mm} (z\neq0; \ 0<q< p\leq 1);\\
 	\end{equation}
 	
 	From equation (\ref{2}), it is clear that if $f$ and $g$ are the two functions, then 
 	\begin{equation}\label{1*}
	D_{p,q}\left( f(z)+g(z)\right) = 	D_{p,q}f(z)+	D_{p,q}g(z)
 	\end{equation}
 	and 
 	\begin{equation}\label{2*}
	D_{p,q}\left( cf(z)\right) =	c D_{p,q}f(z),
 	\end{equation}
 	where $c$ is constant.\\
 	
 	We note that $D_{p,q}f(z)\longrightarrow f^{\prime}(z)$ as $p=1$ and $q\longrightarrow 1-$, where $f^{\prime}$ is the ordinary derivative of the function $f$.\\
 	
 	  In particular, using equation (\ref{2}), the $(p,q)$-derivative of the function $h(z)=z^{n}$ is as follows :
 	\begin{equation}\label{106}
 	D_{p,q}h(z)=[n]_{p,q}z^{n-1}, \hspace{15mm} 
 	\end{equation}
 	where $ [n]_{p,q}$ denotes the $(p,q)$-number and is given as:
 	\begin{equation}\label{3}
 	[n]_{p,q}=\dfrac{p^{n}-q^{n}}{p-q}\hspace{15mm}( 0<q<p\leq 1).\\
 	\end{equation}
 	
 	Since, we note that $[n]_{p,q}\longrightarrow n$ as $p=1$ and $q\longrightarrow 1-$, therefore in view of equation (\ref{106}), $D_{p,q}h(z)\longrightarrow h^{\prime}(z)$ as $p=1$ and $q\longrightarrow 1-$, where $h^{\prime}(z)$ denotes  the ordinary derivative of the function $h(z)$ with respect to $z$. \\
 	
 	Also, the $(p,q)$-integral of the function $f$ on $[0,z]$ is defined as \cite{kang2016some} :
 
 	\begin{equation*}
 	\int_{0}^{z} f(t) d_{p,d}t = (p-q)z \sum_{k=0}^{\infty}\dfrac{q^{k}}{p^{k+1}}f \left( \dfrac{q^{k}}{p^{k+1}}z\right),\\
 \end{equation*}
 where $ \left| \dfrac{q}{p}\right| <1$ and $0<q<p\leq 1$.\\

  In particular, the $(p,q)$-integral of the function $h(z)=z^{n}$ is given by
\begin{equation}\label{4}
\int_{0}^{z}h(t)d_{p,d}t=\dfrac{z^{n+1}}{[n+1]_{p,q}}, \hspace{15mm} 
\end{equation}
where $n \neq -1$ and $ [.]_{p,q}$ is given by equation (\ref{3}).\\

Again, since $[n+1]_{p,q}\longrightarrow n+1$ as $p=1$ and $q\longrightarrow 1-$, therefore for the same choices of $p$ and $q$, equation (\ref{4}) reduces to $\int_{0}^{z}h(t)dt= \dfrac{z^{n+1}}{n+1}$,  which is  the ordinary integral of the function $h(z)$ on  $[0,z]$.\\

 In this paper, we consider the class $\mathcal{A}$ consisting of functions of the following form :
\begin{equation}\label{1}
f(z)=z+\sum_{n=2}^{\infty}a_{n}z^{n}
\end{equation}
and analytic in the open unit disc $\mathbb{U}=\left\lbrace z\in \mathbb{C}:|z|<1\right\rbrace $.\\

Also, using equations (\ref{1*}), (\ref{2*}) and (\ref{106}), we get the $(p,q)$-derivative of the function $f$, given by equation (\ref{1}) as :
\begin{equation}\label{105}
	D_{p,q}f(z)=1+\sum_{n=2}^{\infty}[n]_{p,q}a_{n}z^{n-1} \hspace{15mm} (0<q<p\leq 1)
\end{equation}
where $ [n]_{p,q}$ is given by equation (\ref{3}).\\

For the analytic functions $f$ and $g$ in $\mathbb{U}$, we say that the function $g$ is subordinate to  $f$ in $\mathbb{U}$ \cite{miller2000differential}, and write 
\begin{center}
	$g(z)\prec f(z) \ $  or  $ \ g\prec f$ ,
\end{center}
if there exists a Schwarz function $w$, which is analytic in $\mathbb{U}$ with 
\begin{center}
	$ \ \ \ \ \ \ \ \ \ \ \ w(0)=0$ and $|w(z)|<1, \ \ \ \ \ \ \ \ \ \ \  $
\end{center}
such that 
\begin{equation}\label{0}
	 g(z)=f(w(z))\hspace{5mm} (z\in \mathbb{U}).\\
\end{equation}

\vspace{5mm}

Ma-Minda \cite{Ma} defined the classes of starlike and convex functions, denoted by $\mathcal{S}^{*}(\phi)$ and  $\mathcal{C}(\phi)$, respectively, by using the subordination principle between certain analytic functions. These subclasses are defined as follows:
\begin{equation} \label{6}
\mathcal{S}^{*}(\phi)=\left\lbrace f\in \mathcal{A}:  \dfrac{zf^{\prime}(z)}{f(z)} \prec \phi(z)\right\rbrace \hspace{15mm}
\end{equation}
and
\begin{equation}\label{7}
\mathcal{C}(\phi)=\left\lbrace f\in \mathcal{A} : \left(  1+\dfrac{zf^{\prime\prime}(z)}{f^{\prime}(z)}\right)  \prec  \phi(z) \right\rbrace  ,
\end{equation}
where  the function $\phi(z)$ is analytic in $\mathbb{U}$ with $ \Re (\phi(z))>0$, $\phi(0)=1$ and $\phi^{\prime}(0)>0$.
It is clear that $\mathcal{S}^{*}(\phi)$ and  $\mathcal{C}(\phi)$ are the subclasses of $\mathcal{A}$.\\

The classes of $q$-starlike and  $q$-convex functions, denoted by $\mathcal{S}^{*}_{q}(\phi)$ and $\mathcal{C}_{q}(\phi)$, respectively, are defined by using the subordination principle as \cite{cetinkayafekete}:
\begin{equation}\label{8}
\mathcal{S}^{*}_{q}(\phi)=\left\lbrace f\in \mathcal{A}:  z\dfrac{D_{q}f(z)}{f(z)} \prec \phi(z)\right\rbrace \hspace{10mm}
\end{equation}
and 
\begin{equation}\label{9}
\mathcal{C}_{q}(\phi)=\left\lbrace f\in \mathcal{A}:  \dfrac{D_{q}(zD_{q}f(z))}{D_{q}f(z)} \prec \phi(z)\right\rbrace,\\
\end{equation}
where  the function $\phi(z)$ is analytic in $\mathbb{U}$ with $ \Re (\phi(z))>0$, $\phi(0)=1$ and $\phi^{\prime}(0)>0$. These classes are the subclasses of $\mathcal{A}$.\\

The Feteke-Szeg{\"o} problem \cite{fekete1933bemerkung} is to find the coefficients estimates for second and third coefficients of functions in any class of analytic function having a specified geometric property. In this paper, we introduce the classes of  $(p,q)$-starlike and $(p,q)$-convex functions  by using the $(p,q)$-derivative in terms of the subordination principle . Also, we find the Fekete-Szeg{\"o} inequalities which is obtained by the 
maximizing the absolute value of the coefficient
$|a_{3}-a^{2}_{2}|$ for the functions belonging to these classes, see for example (\cite{DARUS},\cite{FRASI}, \cite{KANAS}) \cite{MISHRA} and \cite{Tang}). Further, the $(p,q)$-Bernardi integral operator for analytic functions, is defined in the open unit disc $\mathbb{U}$ to discuss the application of the results established in this paper.\\
 
 \section{Main Results }
 
 	First, we define the classes of $(p,q)$-starlike functions and $(p,q)$-convex functions, denoted by $\mathcal{S}^{*}_{p,q}(\phi)$ and  $\mathcal{C}_{p,q}(\phi)$, respectively, in terms of the  subordination principle by taking the $(p,q)$-derivative in place of $q$-derivative in the respective definitions of the classes of $q$-starlike and $q$-convex functions.\\
 
 The respective definitions of the classes $\mathcal{S}^{*}_{p,q}(\phi)$ and $\mathcal{C}_{p,q}(\phi)$ are as follows:
 \\	
 	\begin{defn}
 	The function $f\in \mathcal{A}$ is said to be  $(p,q)$-starlike if it satisfies the following subordination:
 	
 	\begin{equation}\label{21}
 \dfrac{zD_{p,q}f(z)}{f(z)} \prec \phi(z) \hspace{15mm}(  0<q<p\leq1),
 	\end{equation}
 	where  the function $\phi(z)$ is analytic in $\mathbb{U}$ with $ \Re (\phi(z))>0$, $\phi(0)=1$ and $\phi^{\prime}(0)>0$.\\
 	
 \end{defn}
	\vspace{5mm}
\begin{defn}
	The function $f\in \mathcal{A}$ is said to be  $(p,q)$-convex  if it satisfies the following subordination:
	\begin{equation}\label{22}
	\dfrac{D_{p,q}(zD_{p,q}f(z))}{D_{p,q}f(z)} \prec \phi(z) \hspace{15mm}( 0<q<p\leq 1),
	\end{equation}
	where  the function $\phi(z)$ is analytic in $\mathbb{U}$ with $ \Re (\phi(z))>0$, $\phi(0)=1$ and $\phi^{\prime}(0)>0$.
\end{defn}
\vspace{5mm}
\begin{figure}[h]
	\centering
	\includegraphics[width=0.50\linewidth]{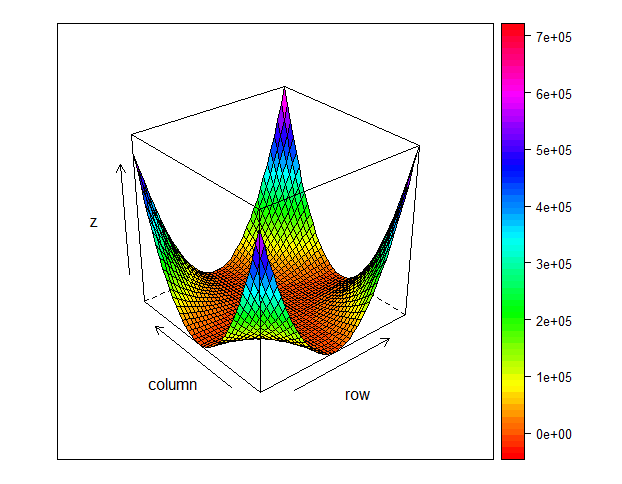}
	\caption{The class $\mathcal{S}^{*}_{0.2,0.5}\left( \dfrac{1+z}{1-z}\right) $ for the complex number $z=x+iy, \hspace{2mm}x,y\in \mathbb{R}$.}
	\label{fig:rplot}
\end{figure}

\begin{figure}[h]
	\centering
	\includegraphics[width=0.50\linewidth]{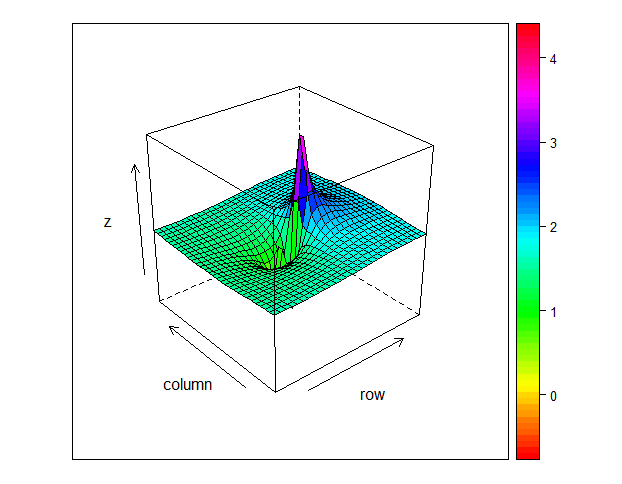}
	\caption{The class $\mathcal{C}_{0.2,0.5}\left( \dfrac{1+z}{1-z}\right) $ for the complex number $z=x+iy, \hspace{2mm}x,y\in \mathbb{R}$.}
	\label{fig:rplot}
\end{figure}
\begin{rem}
	 We note that, for $p=1$ the classes $\mathcal{S}^{*}_{p,q}(\phi)$ and $\mathcal{C}_{p,q}(\phi)$, reduce to the classes $\mathcal{S}^{*}_{q}(\phi)$ and $\mathcal{C}_{q}(\phi)$, which are defined by equations (\ref{8}) and (\ref{9}), respectively. Again, for $p=1$ and $q\longrightarrow 1-$, the classes  $\mathcal{S}^{*}_{p,q}(\phi)$ and $\mathcal{C}_{p,q}(\phi)$ reduce to the classes $\mathcal{S}^{*}(\phi)$, defined by equation (\ref{6})  and $\mathcal{C}(\phi)$, defined by equation (\ref{7}), respectively. \\
\end{rem}

 First of all, we need to mention the following lemma \cite{Ma}:\\
 
 \begin{lem}
 	If $p(z)=1+c_{1}z+c_{2}z^{2}+\dots$ is a function with $\Re(p(z))>0$ and $\mu\in \mathbb{C}$, then
 	\begin{equation*}
	|c_{2}-\mu c^{2}_{1}|\leq 2 \max\left\lbrace 1;|2\mu -1|\right\rbrace .\\
 	\end{equation*} 
 	
 	The result is sharp for giving two choices of the function $p(z)$ as follows:
 	\begin{center}
$p(z)=\dfrac{1+z^{2}}{1-z^{2}}$ and $p(z)=\dfrac{1+z}{1-z}$.
 	\end{center}
 \end{lem}

Now, we investigate the Feteke-Szeg{\"o} inequality of the class  $\mathcal{S}^{*}_{p,q}(\phi)$ in the following result:\\
\begin{thm}
	Let $\phi (z)=1+b_{1}z+b_{2}z^{2} \dots $, with $b_{1}\neq 0$. If $f$, given by equation (\ref{1}), belongs to the class $\mathcal{S}^{*}_{p,q}(\phi)$, then 
	\begin{equation}\label{23}
	|a_{3}-\mu a^{2}_{2}|\leq \dfrac{|b_{1}|}{[3]_{p,q}-1} \max \left\lbrace 1; \left|\dfrac{ b_{2}}{b_{1}}+\dfrac{b_{1}}{[2]_{p,q}-1}\left(1-\dfrac{[3]_{p,q}-1}{[2]_{p,q}-1} \right) \mu \right|  \right\rbrace ,
	\end{equation}
	where $b_{1}, b_{2}, \dots \in \mathbb{R} $, $\mu \in \mathbb{C}$ and $0<q<p\leq 1$. The result is sharp.\\
\end{thm}

\begin{proof}[\textbf{Proof}]
Let $f \in \mathcal{S}^{*}_{p,q}(\phi)$, then in view of Definition 2.1, the function $f$ satisfies the subordination (\ref{21}). Thus, by using equation (\ref{0}), there is a Schwarz function $w$ such that \\
\begin{equation}\label {25}
 \dfrac{zD_{p,q}f(z)}{f(z)}=\phi (w(z)).\\
\end{equation}

We define the function
\begin{equation}\label{111}
p(z)=1+c_{1}z+c_{2}z^{2}+\dots
\end{equation}
 in terms of the function $w(z)$ as :
\begin{equation*}
 p(z)=\dfrac{1+w(z)}{1-w(z)}, 
\end{equation*}
which gives
\begin{equation}\label{29}
 w(z)=\dfrac{p(z)-1}{p(z)+1}.
\end{equation}

Using equations (\ref{111}) and (\ref{29}), we get
  \begin{equation}\label {26}
  \phi (w(z))=\phi\left( \dfrac{c_{1}z+c_{2}z^{2}+\dots}{2+c_{1}z+c_{2}z^{2}+\dots}\right) = \phi \left( \dfrac{1}{2}\left[c_{1}z+\left( c_{2}-\dfrac{1}{2}c^{2}_{1}\right) z^{2}+\left( c_{3}-c_{1}c_{2}+\dfrac{c^{3}_{1}}{4}\right) z^{3}+\dots\right]\right).\\
  \end{equation}
  
 Since $\phi (z)=1+b_{1}z+b_{2}z^{2} \dots $, therefore, equation (\ref{26}) gives
 \begin{equation}\label{27}
\phi (w(z))=1+\dfrac{b_{1}c_{1}}{2}z+\left[ \dfrac{b_{1}}{2}\left(c_{2}-\dfrac{c^{2}_{1}}{2}\right)+\dfrac{b_{2}c^{2}_{1}}{4} \right]z^{2}+\dots. \\
 \end{equation}
 
 Now, using equations (\ref{1}) and (\ref{105}), we get
 \begin{equation}\label{24}
 \dfrac{zD_{p,q}f(z)}{f(z)}=\dfrac{z+\sum_{n=2}^{\infty}[n]_{p,q}a_{n}z^{n}}{z+\sum_{n=2}^{\infty}a_{n}z^{n}}=1+([2]_{p,q}-1)a_{2}z+\bigg(([3]_{p,q}-1)a_{3}-([2]_{p,q}-1)a^{2}_{2} \bigg)z^{2} +\dots ~ .\\
 \end{equation}
 
 Using equations (\ref{27}) and (\ref{24}) in equation (\ref{25}), then comparing the coefficients of $z$ and $z^{2}$ from the both sides of the resultant equation and simplifying, we get
 \begin{equation}\label{30}
a_{2}=\dfrac{b_{1}c_{1}}{2 ([2]_{p,q}-1)}
 \end{equation}
 and 
 \begin{equation}\label{31}
 a_{3}=\dfrac{b_{1}}{2 ([3]_{p,q}-1)}\left[ c_{2}-\dfrac{1}{2}\left( 1-\dfrac{b_{2}}{b_{1}}-\dfrac{b_{1}}{[2]_{p,q}-1}\right)c^{2}_{1} \right] .\\
 \end{equation}
 
 Next, for $\mu \in \mathbb{C}$, using equations (\ref{30}) and (\ref{31}), we have
 \begin{equation}\label{33}
a_{3}-\mu a^{2}_{2}=\dfrac{b_{1}}{2 ([3]_{p,q}-1)}\left[c_{2}- \dfrac{1}{2}\left( 1-\dfrac{b_{2}}{b_{1}}-\dfrac{b_{1}}{[2]_{p,q}-1}\left(1-\dfrac{[3]_{p,q}-1}{[2]_{p,q}-1} \mu \right) \right)c^{2}_{1} \right].\\
 \end{equation}
 
If we take 
 \begin{equation}\label{34}
v=\dfrac{1}{2}\left( 1-\dfrac{b_{2}}{b_{1}}-\dfrac{b_{1}}{[2]_{p,q}-1}\left(1-\dfrac{[3]_{p,q}-1}{[2]_{p,q}-1} \mu \right) \right),
 \end{equation}
 
 then, from equation  (\ref{33}), we get
 \begin{equation}\label{35}
|a_{3}-\mu a^{2}_{2}|=\dfrac{|b_{1}|}{2 ([3]_{p,q}-1)} |c_{2}-vc^{2}_{1}|.\\
 \end{equation}
 
 Hence, by applying Lemma 2.1, equation (\ref{35}), gives the Feteke-Szeg{\"o} inequality, given by equation (\ref{23}), for the class $\mathcal{S}^{*}_{p,q}(\phi)$.\\
 
 Further, our result is sharp, that is, the equality holds, when $p(z)=p_{1}(z)=\dfrac{1+z}{1-z}=1
 +2z+2z^{2}+\dots$ and equation (\ref{25}), gives
 \begin{equation}\label{107}
\dfrac{zD_{p,q}f(z)}{f(z)}=\phi \left(\dfrac{p_{1}(z)-1}{p_{1}(z)+1} \right)=\phi(z)=1+b_{1}z+b_{2}z^{2} \dots .\\
 \end{equation}
 
 Then, by comparing equations (\ref{27}) and (\ref{107}), we have $c_{1}=2$ and $c_{2}=2$, then equation (\ref{33}) gives the equality sign in the place of inequality in assertion  (\ref{23}).\\
 
 Similarly, for $p(z)=p_{2}(z)=\dfrac{1+z^{2}}{1-z^{2}}=1+2z^{2}+\dots$, equation (\ref{25}) gives
 \begin{equation}\label{108}
\dfrac{zD_{p,q}f(z)}{f(z)}= \phi \left(\dfrac{p_{2}(z)-1}{p_{2}(z)+1} \right)=\phi (z^{2})=1+b_{1}z^{2}+\dots .\\
 \end{equation}
 
 Then, by comparing equations (\ref{27}) and (\ref{108}), we have $c_{1}=0$ and $c_{2}=2$ and hence  equation (\ref{33}) gives the equality sign in the place of  inequality  in assertion (\ref{23}).\\
\end{proof}

Taking $p=1$ and $q\longrightarrow 1-$ in Theorem 2.1, we get the following corollary \cite{cetinkayafekete}:
\begin{cor}
	Let $\phi (z)=1+b_{1}z+b_{2}z^{2} \dots $, with $b_{1}\neq 0$. If $f$ given by equation (\ref{1}) belongs to the class $\mathcal{S}^{*}(\phi)$, then 
\begin{equation*}
|a_{3}-\mu a^{2}_{2}|\leq \dfrac{|b_{1}|}{2} \max \left\lbrace 1; \left|\dfrac{ b_{2}}{b_{1}}+b_{1}\left(1-2 \mu \right) \right|  \right\rbrace ,
\end{equation*}
where $b_{1}, b_{2}, \dots \in \mathbb{R} $ and $\mu \in \mathbb{C}$ . The result is sharp.\\
\end{cor}
\begin{rem}
	For $p=1$, inequality (\ref{23}), gives the Feteke-Szeg{\"o} inequality  \cite{cetinkayafekete} for the class  $\mathcal{S}^{*}_{q}(\phi)$ .\\
\end{rem}

Next, we investigate the Feteke-Szeg{\"o} inequality for the class  $\mathcal{C}_{p,q}(\phi)$ in the following result:\\

\begin{thm}
	Let $\phi (z)=1+b_{1}z+b_{2}z^{2} \dots $ with $b_{1}\neq 0$. If $f$,  given by equation (\ref{1}), belongs to the class $\mathcal{C}_{p,q}(\phi)$, then 
	\begin{equation}\label{44}
	|a_{3}-\mu a^{2}_{2}|\leq \dfrac{|b_{1}|}{[3]_{p,q}([3]_{p,q}-1)} \max \left\lbrace 1; \left|\dfrac{ b_{2}}{b_{1}}+\dfrac{b_{1}}{[2]_{p,q}-1}\left(1-\dfrac{[3]_{p,q}([3]_{p,q}-1)}{[2]^{2}_{p,q}([2]_{p,q}-1)} \right) \mu \right|  \right\rbrace ,
	\end{equation}
	where $b_{1}, b_{2}, \dots \in \mathbb{R} $, $\mu \in \mathbb{C}$ and $0<q<p\leq 1$. The result is sharp.\\
\end{thm}

\begin{proof}[\textbf{Proof}]
	Let $f \in \mathcal{C}_{p,q}(\phi)$, then in view of Definition 2.2 the function $f$ satisfies the subordination (\ref{22}), thus, by using equation (\ref{0}), there is a Schwarz function $w$ such that \\
	\begin{equation}\label {36}
	\dfrac{D_{p,q}(zD_{p,q}f(z))}{D_{p,q}f(z)}=\phi (w(z)),
	\end{equation}
 where $w$ is given by equation (\ref{29}) and $\phi(w(z))$ is given by equation (\ref{27}).\\
     
  Using equations (\ref{1}) and (\ref{105}), we get
 \begin{equation}\label{37}
\begin{split}
&\dfrac{D_{p,q}(zD_{p,q}f(z))}{D_{p,q}f(z)}=\dfrac{z+\sum_{n=2}^{\infty}[n]^{2}_{p,q}a_{n}z^{n}}{z+\sum_{n=2}^{\infty}[n]_{p,q} a_{n}z^{n}}\\
&\hspace{26mm}=1+[2]_{p,q}([2]_{p,q}-1)a_{2}z+ \bigg([3]_{p,q}([3]_{p,q}-1)a_{3}-[2]^{2}_{p,q}([2]_{p,q}-1)a^{2}_{2} \bigg)z^{2} +\dots ~.\\
\end{split}
 \end{equation}
 
Comparing the coefficients of $z$ and $z^{2}$ in equations (\ref{27}) and (\ref{37}) and simplifying, we obtain 
\begin{equation}\label{38}
	a_{2}=\dfrac{b_{1}c_{1}}{2 [2]_{p,q}([2]_{p,q}-1)}
\end{equation}
and 
\begin{equation}\label{39}
	a_{3}=\dfrac{b_{1}}{2 [3]_{p,q}([3]_{p,q}-1)}\left[ c_{2}-\dfrac{1}{2}\left( 1-\dfrac{b_{2}}{b_{1}}-\dfrac{b_{1}}{[2]_{p,q}-1}\right)c^{2}_{1} \right] .\\
\end{equation}

Next, for $\mu \in \mathbb{C}$, equations (\ref{38}) and (\ref{39}), give
\begin{equation}\label{40}
	a_{3}-\mu a^{2}_{2}=\dfrac{b_{1}}{2 [3]_{p,q}([3]_{p,q}-1)}\left[c_{2}- \dfrac{1}{2}\left( 1-\dfrac{b_{2}}{b_{1}}-\dfrac{b_{1}}{[2]_{p,q}-1}\left(1-\dfrac{[3]_{p,q}([3]_{p,q}-1)}{[2]^{2}_{p,q}([2]_{p,q}-1)} \mu \right) \right) c^{2}_{1}\right] .\\
\end{equation}

If we take 
\begin{equation}\label{41}
	v=\dfrac{1}{2}\left( 1-\dfrac{b_{2}}{b_{1}}-\dfrac{b_{1}}{[2]_{p,q}-1}\left(1-\dfrac{[3]_{p,q}([3]_{p,q}-1)}{[2]^{2}_{p,q}([2]_{p,q}-1)} \mu \right) \right),
\end{equation}
then using equations (\ref{40}) and (\ref{41}), we get
\begin{equation}\label{42}
	|a_{3}-\mu a^{2}_{2}|=\dfrac{|b_{1}|}{2 [3]_{p,q}([3]_{p,q}-1)} |c_{2}-vc^{2}_{1}|.
\end{equation}

Now, by applying Lemma 2.1, equation (\ref{42}), gives the Feteke-Szeg{\"o} inequality, given by equation (\ref{44}) for the class $\mathcal{C}_{p,q}(\phi)$. \\

 Further, our result is sharp, when $p(z)=p_{1}(z)=\dfrac{1+z}{1-z}=1
+2z+2z^{2}+\dots$  and  equation (\ref{36}), gives
\begin{equation}\label{109}
\dfrac{D_{p,q}(zD_{p,q}f(z))}{D_{p,q}f(z)}=\phi \left(\dfrac{p_{1}(z)-1}{p_{1}(z)+1} \right)=\phi(z)=1+b_{1}z+b_{2}z^{2} \dots~.\\
\end{equation}

Then, by comparing equations (\ref{27}) and (\ref{109}), we have $c_{1}=2$ and $c_{2}=2$ and hence  equation (\ref{40}) gives the equality sign in the place of inequality in assertion (\ref{44}) .\\

Similarly, when $p(z)=p_{2}(z)=\dfrac{1+z^{2}}{1-z^{2}}=1+2z^{2}+\dots$, equation (\ref{36}) gives
\begin{equation}\label{110}
\dfrac{D_{p,q}(zD_{p,q}f(z))}{D_{p,q}f(z)}= \phi \left(\dfrac{p_{2}(z)-1}{p_{2}(z)+1} \right)=\phi (z^{2})=1+b_{1}z^{2}+\dots,
\end{equation}
then, by comparing equations (\ref{27}) and (\ref{110}), we have $c_{1}=0$ and $c_{2}=2$ and hence  equation (\ref{40}) gives the equality sign in the place of inequality in assertion (\ref{44}).\\

\end{proof}
Taking $p=1$ and $q\longrightarrow 1-$ in Theorem 2.2, we get the following corollary \cite{cetinkayafekete}:
\begin{cor}
		Let $\phi (z)=1+b_{1}z+b_{2}z^{2} \dots $, with $b_{1}\neq 0$. If $f$ given by equation (\ref{1}) belongs to the class $\mathcal{C}(\phi)$, then 
	\begin{equation}
	|a_{3}-\mu a^{2}_{2}|\leq \dfrac{|b_{1}|}{6} \max \left\lbrace 1; \left|\dfrac{ b_{2}}{b_{1}}+b_{1}\left(1-\dfrac{3}{2}\mu \right)  \right|  \right\rbrace ,
	\end{equation}
	where $b_{1}, b_{2}, \dots \in \mathbb{R} $ and $\mu \in \mathbb{C}$ . The result is sharp.\\
\end{cor}
\begin{rem}
	For $p=1$, inequality (\ref{44}) gives the Feteke-Szeg{\"o} inequality for the class  $\mathcal{C}_{q}(\phi)$ \cite{cetinkayafekete}.\\
\end{rem}

In the next section, we discuss the coefficient bounds of the first and third coefficients of the functions belonging to the classes  $\mathcal{S}^{*}_{p,q}(\phi)$ and $\mathcal{C}_{p,q}(\phi)$.\\

\section{Coefficient bounds}

In this section, we estimate the coefficient bounds for the  coefficients of $z$ and $z^{2}$ of $(p,q)$-starlike and $(p,q)$-convex functions.\\

First, we need to mention the following lemma \cite{Ma} :\\
\begin{lem}
	If $p(z)=1+c_{1}z+c_{2}z^{2}+\dots$ is a function with $\Re(p(z))>0$, then 
	\begin{eqnarray}\label{10}
	|c_{2}-v c^{2}_{1}|\leq \left\{
	\begin{array}{lll}
	-4v+2, &if ~~ v\leq 0;
	\\
	2, & if ~~ 0\leq v \leq 1;
	\\
	4v-2, & if~~ v\geq 1.
	\end{array}\right.
	\end{eqnarray}\\
	Also, 
	the above upper bound is sharp, and it can be improved as follows when $0<v<1$:
\begin{equation}\label{100}
|c_{2}-v c^{2}_{1}|+v|c_{1}|^{2} \leq 2 \ \ \ \left( 0<v\leq \dfrac{1}{2}\right)
\end{equation}
and
\begin{equation}\label{101}
|c_{2}-v c^{2}_{1}|+(1-v)|c_{1}|^{2} \leq 2 \ \ \ \left(  \dfrac{1}{2}\leq v < 1 \right).\\
\end{equation}
\end{lem}
\vspace{5mm}
Now, we establish the following result for estimation of the coefficient bound for the functions belonging to the class  $\mathcal{S}^{*}_{p,q}(\phi)$ :\\

\begin{thm}
	Let $\phi (z)=1+b_{1}z+b_{2}z^{2} \dots $  with $b_{1}> 0$ and $b_{2}\geq 0$. Let
	\begin{align}
	\sigma _{1}&=\dfrac{([2]_{p,q}-1)b^{2}_{1}+([2]_{p,q}-1)^{2}(b_{2}-b_{1})}{([3]_{p,q}-1)b^{2}_{1}},\label{50}\\
	\sigma _{2}&=\dfrac{([2]_{p,q}-1)b^{2}_{1}+([2]_{p,q}-1)^{2}(b_{2}+b_{1})}{([3]_{p,q}-1)b^{2}_{1}},\label{51}\\
	\sigma _{3}&=\dfrac{([2]_{p,q}-1)b^{2}_{1}+([2]_{p,q}-1)^{2} b_{2}}{([3]_{p,q}-1)b^{2}_{1}}.\label{53}
	\end{align}
	\vspace{5mm}
	
	If $f$, given by equation (\ref{1}), belongs to the class $\mathcal{S}^{*}_{p,q}(\phi)$, then
		\begin{eqnarray}\label{58}
|a_{3}-\mu a^{2}_{2}|\leq \left\{
	\begin{array}{lll}
\dfrac{b_{2}}{[3]_{p,q}-1}+\dfrac{b^{2}_{1}}{[2]_{p,q}-1}\left(\dfrac{1}{[3]_{p,q}-1}-\dfrac{\mu}{[2]_{p,q}-1}\right)  , &if ~~ \mu \leq \sigma_{1};
	\\
	\\
	\dfrac{b_{1}}{[3]_{p,q}-1}, & if ~~ \sigma_{1}\leq \mu \leq \sigma_{2};
	\\
	\\
\dfrac{-b_{2}}{[3]_{p,q}-1}-\dfrac{b^{2}_{1}}{[2]_{p,q}-1}\left(\dfrac{1}{[3]_{p,q}-1}-\dfrac{\mu}{[2]_{p,q}-1}\right) , & if~~ \mu \geq \sigma_{2}.
	\end{array}\right.
	\end{eqnarray}\\ 
	
	Further, if $\sigma_{1} < \mu \leq \sigma_{3}$, then 
	\begin{equation}\label{55}
|a_{3}-\mu a^{2}_{2}|+\dfrac{([2]_{p,q}-1)^{2}}{([3]_{p,q}-1)b^{2}_{1}}\left[ b_{1}-b_{2}-\dfrac{b^{2}_{1}}{[2]_{p,q}-1}\left( 1-\dfrac{[3]_{p,q}-1}{[2]_{p,q}-1}\mu \right)|a_{2}|^{2} \right] \leq \dfrac{b_{1}}{[3]_{p,q}-1}.\\
	\end{equation}
	and if 
	$\sigma_{3}\leq \mu < \sigma_{2}$, then 
	\begin{equation}\label{56}
	|a_{3}-\mu a^{2}_{2}|+\dfrac{([2]_{p,q}-1)^{2}}{([3]_{p,q}-1)b^{2}_{1}}\left[ b_{1}+b_{2}+\dfrac{b^{2}_{1}}{[2]_{p,q}-1}\left( 1-\dfrac{[3]_{p,q}-1}{[2]_{p,q}-1}\mu \right)|a_{2}|^{2} \right] \leq \dfrac{b_{1}}{[3]_{p,q}-1}.\\
	\end{equation}
	
\end{thm}
\vspace{5mm}
\begin{proof}[\textbf{Proof}]
	For $v\leq 0$, equation (\ref{34}) gives
\begin{equation*}
\mu \leq \dfrac{([2]_{p,q}-1)b^{2}_{1}+([2]_{p,q}-1)^{2}(b_{2}-b_{1})}{([3]_{p,q}-1)b^{2}_{1}}.\\
\end{equation*}

Let  $\dfrac{([2]_{p,q}-1)b^{2}_{1}+([2]_{p,q}-1)^{2}(b_{2}-b_{1})}{([3]_{p,q}-1)b^{2}_{1}}=\sigma _{1}$, then from the above relation, we have $\mu \leq \sigma_{1}$.\\

Let $p(z)$ be a function, given by equation (\ref{111}), with $\Re\left( p(z)\right)>0 $ and $f(z)$, given by equation (\ref{1}), be a member of the class  $\mathcal{S}^{*}_{p,q}(\phi)$, then equation (\ref{35}) holds. Thus using Lemma 3.1 for $v\leq 0$ in  equation (\ref{35}), we get
\begin{equation*}
|a_{3}-\mu a^{2}_{2}|\leq \dfrac{b_{1}}{2 ([3]_{p,q}-1)} (-4v+2),
\end{equation*}
which on using equation (\ref{34}), gives

\begin{equation}\label{57}
|a_{3}-\mu a^{2}_{2}|\leq \dfrac{b_{1}}{ [3]_{p,q}-1}\left(\dfrac{b_{2}}{b_{1}}+\dfrac{b_{1}}{[2]_{p,q}-1}\left(1-\dfrac{[3]_{p,q}-1}{[2]_{p,q}-1} \mu \right) \right),
\end{equation}
 where $\mu \leq \sigma_{1}$.\\
 
 Simplifying the right hand side of inequality  (\ref{57}), we get the first inequality of assertion (\ref{58}).\\

Again, if we take  $0\leq v\leq 1$, then equation (\ref{34}), gives

\begin{equation*}
\sigma_{1} \leq \mu \leq \dfrac{([2]_{p,q}-1)b^{2}_{1}+([2]_{p,q}-1)^{2}(b_{2}+b_{1})}{([3]_{p,q}-1)b^{2}_{1}},
\end{equation*}
where $\sigma_{1}$ is given by equation (\ref{50}).\\

Let  $\dfrac{([2]_{p,q}-1)b^{2}_{1}+([2]_{p,q}-1)^{2}(b_{2}-b_{1})}{([3]_{p,q}-1)b^{2}_{1}}=\sigma _{2}$, then from the above relation, we have $ \sigma_{1} \leq \mu \leq \sigma_{2}$.\\

Now, using Lemma 3.1 for $0\leq v\leq 1$ in equation (\ref{35}), we get
\begin{equation*}
|a_{3}-\mu a^{2}_{2}|\leq \dfrac{b_{1}}{ [3]_{p,q}-1},
\end{equation*}
which gives the second  inequality of assertion  (\ref{58}).\\

Next, if we take $ v\geq 1$, then equation (\ref{34}), gives that $\mu\geq \sigma_{2} $.\\

Now, using Lemma 3.1, for $ v\geq 1$  in equation (\ref{35}), we get
\begin{equation*}
|a_{3}-\mu a^{2}_{2}|\leq \dfrac{b_{1}}{2 ([3]_{p,q}-1)} (4v-2),
\end{equation*}
which on using equation (\ref{34}), gives
\begin{equation}\label{60}
|a_{3}-\mu a^{2}_{2}| \leq \dfrac{b_{1}}{[3]_{p,q}-1}\left(-\dfrac{b_{2}}{b_{1}}-\dfrac{b_{1}}{[2]_{p,q}-1}\left(1-\dfrac{[3]_{p,q}-1}{[2]_{p,q}-1} \mu \right) \right).\\
\end{equation}

Inequality  (\ref{60}) gives the third inequality of assertion (\ref{58}).\\

Further, if $0<v\leq \dfrac{1}{2}$, then using equation (\ref{34}), we have
\begin{equation*}
0	<\dfrac{1}{2}\left( 1-\dfrac{b_{2}}{b_{1}}-\dfrac{b_{1}}{[2]_{p,q}-1}\left(1-\dfrac{[3]_{p,q}-1}{[2]_{p,q}-1} \mu \right) \right)\leq \dfrac{1}{2},
\end{equation*}
which on simplifying, gives 
\begin{equation}\label{112}
\sigma_{1}	<	\mu \leq 	\dfrac{([2]_{p,q}-1)b^{2}_{1}+([2]_{p,q}-1)^{2} b_{2}}{([3]_{p,q}-1)b^{2}_{1}},
\end{equation}
where $\sigma _{1}$ is given by equation (\ref{50}).\\

Let $\dfrac{([2]_{p,q}-1)b^{2}_{1}+([2]_{p,q}-1)^{2} b_{2}}{([3]_{p,q}-1)b^{2}_{1}}=\sigma _{3}$, then from relation (\ref{112}), we have $ \sigma_{1} < \mu \leq \sigma_{3}$.\\

Now, using equations (\ref{30}) and (\ref{50}), we get
\begin{equation}\label {63}
\begin{split}
&|a_{3}-\mu a^{2}_{2}|+(\mu - \sigma_{1})|a_{2}|^{2}=|a_{3}-\mu a^{2}_{2}|\\
&\hspace{40mm}+\left( \mu - \dfrac{([2]_{p,q}-1)b^{2}_{1}+([2]_{p,q}-1)^{2}(b_{2}-b_{1})}{([3]_{p,q}-1)b^{2}_{1}} \right) \dfrac{b^{2}_{1}|c_{1}|^{2}}{4([2]_{p,q}-1)^{2}},\\
\end{split}
\end{equation}
which on using equation (\ref{35}), we get 
\begin{equation}\label{113}
|a_{3}-\mu a^{2}_{2}|+(\mu - \sigma_{1})|a_{2}|^{2}=\dfrac{b_{1}}{2 ([3]_{p,q}-1)}\left( |c_{2}-vc^{2}_{1}|+\dfrac{1}{2}\left( 1 - \dfrac{b_{2}}{b_{1}}- \dfrac{b_{1}}{[2]_{p,q}-1}\left( 1-\dfrac{[3]_{p,q}-1}{[2]_{p,q}-1}\mu \right)\right)  |c_{1}|^{2}\right).\\
\end{equation}

Using equation (\ref{34}) in equation (\ref{113}), we get
\begin{equation*}
|a_{3}-\mu a^{2}_{2}|+(\mu - \sigma_{1})|a_{2}|^{2}=\dfrac{b_{1}}{ [3]_{p,q}-1}\left( \dfrac{1}{2}\left(  |c_{2}-vc^{2}_{1}|+v|c_{1}|^{2}  \right) \right),
\end{equation*}
which in view of inequality (\ref{100}), gives
\begin{equation}\label{65}
|a_{3}-\mu a^{2}_{2}|+(\mu - \sigma_{1})|a_{2}|^{2} \leq \dfrac{b_{1}}{ [3]_{p,q}-1}.\\
\end{equation}

Now, using inequality (\ref{65}) in equation (\ref{63}), we get
\begin{equation*}
|a_{3}-\mu a^{2}_{2}|+\left( \mu - \dfrac{([2]_{p,q}-1)b^{2}_{1}+([2]_{p,q}-1)^{2}(b_{2}-b_{1})}{([3]_{p,q}-1)b^{2}_{1}} \right)|a_{2}|^{2}\leq \dfrac{b_{1}}{ [3]_{p,q}-1},
\end{equation*} 
where $ \sigma_{1} < \mu \leq \sigma_{3}$.\\

Simplifying the above inequality, we obtain the assertion (\ref{55}).\\

Similarly, if  $\dfrac{1}{2}\leq v < 1$, then using equation (\ref{34}), we get $ \sigma_{3} \leq \mu < \sigma_{2}$, where $ \sigma_{2}$ and $\sigma_{3}$ are given by equations (\ref{51}) and (\ref{53}), respectively.  \\

Now, using equations (\ref{30}) and (\ref{51}), we get
\begin{equation}\label{64}
\begin{split}
&|a_{3}-\mu a^{2}_{2}|+(\sigma_{2}-\mu )|a_{2}|^{2}=|a_{3}-\mu a^{2}_{2}| \\
&\hspace{40mm}+\left( \dfrac{([2]_{p,q}-1)b^{2}_{1}+([2]_{p,q}-1)^{2}(b_{2}+b_{1})}{([3]_{p,q}-1)b^{2}_{1}} -\mu\right)  \dfrac{b^{2}_{1}|c_{1}|^{2}}{4([2]_{p,q}-1)^{2}}, \\
\end{split}
\end{equation}

Using equation (\ref{35}) in equation (\ref{64}), we get
\begin{equation}\label {103}
|a_{3}-\mu a^{2}_{2}|+(\sigma_{2}-\mu )|a_{2}|^{2}=\dfrac{b_{1}}{2 ([3]_{p,q}-1)}\left( |c_{2}-vc^{2}_{1}|+\dfrac{1}{2}\left( 1+\dfrac{b_{2}}{b_{1}}+\dfrac{b_{1}}{[2]_{p,q}-1}\left( 1-\dfrac{[3]_{p,q}-1}{[2]_{p,q}-1} \mu\right) \right) |c_{1}|^{2}\right),
\end{equation}

which, on using equation (\ref{34}) gives
\begin{equation}\label{80}
|a_{3}-\mu a^{2}_{2}|+(  \sigma_{2}-\mu)|a_{2}|^{2}=\dfrac{b_{1}}{ [3]_{p,q}-1}\left( \dfrac{1}{2}\left(  |c_{2}-vc^{2}_{1}|+(1-v)|c_{1}|^{2}  \right) \right).\\
\end{equation}

Now, since $\dfrac{1}{2}\leq v < 1$, therefore using inequality (\ref{101}) of Lemma 3.1, equation (\ref{80}) gives

\begin{equation}\label{81}
|a_{3}-\mu a^{2}_{2}|+(\sigma_{2}-\mu)|a_{2}|^{2} \leq \dfrac{b_{1}}{ [3]_{p,q}-1}.\\
\end{equation}

Using inequality (\ref{81}) in  equation (\ref{64}), we get
\begin{equation*}
|a_{3}-\mu a^{2}_{2}|+\left(  \dfrac{([2]_{p,q}-1)b^{2}_{1}+([2]_{p,q}-1)^{2}(b_{2}+b_{1})}{([3]_{p,q}-1)b^{2}_{1}}- \mu \right) |a_{2}|^{2}\leq \dfrac{b_{1}}{ [3]_{p,q}-1},
\end{equation*}
where $ \sigma_{3} \leq \mu <\sigma_{2}$.\\

Finally, on simplifying the above inequality, we obtain the assertion (\ref{56}).\\
\end{proof}

Taking $p=1$ in Theorem 3.1, we get the following corollary for the class $\mathcal{S}^{*}_{q}(\phi)$:\\

\begin{cor}
	Let $\phi (z)=1+b_{1}z+b_{2}z^{2} \dots $ with $b_{1}> 0$ and $b_{2}\geq 0$. Let
\begin{align}
\sigma _{1}&=\dfrac{([2]_{q}-1)b^{2}_{1}+([2]_{q}-1)^{2}(b_{2}-b_{1})}{([3]_{q}-1)b^{2}_{1}},\\
\sigma _{2}&=\dfrac{([2]_{q}-1)b^{2}_{1}+([2]_{q}-1)^{2}(b_{2}+b_{1})}{([3]_{q}-1)b^{2}_{1}},\\
\sigma _{3}&=\dfrac{([2]_{q}-1)b^{2}_{1}+([2]_{q}-1)^{2} b_{2}}{([3]_{q}-1)b^{2}_{1}}.
\end{align}
If $f$, given by equation (\ref{1}),  belongs to the class $S^{*}_{q}(\phi)$, then
\begin{eqnarray}
|a_{3}-\mu a^{2}_{2}|\leq \left\{
\begin{array}{lll}
\dfrac{b_{2}}{[3]_{q}-1}+\dfrac{b^{2}_{1}}{[2]_{q}-1}\left(\dfrac{1}{[3]_{q}-1}-\dfrac{\mu}{[2]_{q}-1}\right)  , &if ~~ \mu \leq \sigma_{1};
\\
\\
\dfrac{b_{1}}{[3]_{q}-1}, & if ~~ \sigma_{1}\leq \mu \leq \sigma_{2};
\\
\\
\dfrac{-b_{2}}{[3]_{q}-1}-\dfrac{b^{2}_{1}}{[2]_{q}-1}\left(\dfrac{1}{[3]_{q}-1}-\dfrac{\mu}{[2]_{q}-1}\right) , & if~~ \mu \geq \sigma_{2}.
\end{array}\right.
\end{eqnarray}\\ 

Further, if $\sigma_{1}< \mu \leq \sigma_{3}$, then 
\begin{equation}
|a_{3}-\mu a^{2}_{2}|+\dfrac{([2]_{q}-1)^{2}}{([3]_{q}-1)b^{2}_{1}}\left[ b_{1}-b_{2}-\dfrac{b^{2}_{1}}{[2]_{q}-1}\left( 1-\dfrac{[3]_{q}-1}{[2]_{q}-1}\mu \right)|a_{2}|^{2} \right] \leq \dfrac{b_{1}}{[3]_{q}-1}\\
\end{equation}
and if 
$\sigma_{3}\leq \mu < \sigma_{2}$, then 
\begin{equation}
|a_{3}-\mu a^{2}_{2}|+\dfrac{([2]_{q}-1)^{2}}{([3]_{q}-1)b^{2}_{1}}\left[ b_{1}+b_{2}+\dfrac{b^{2}_{1}}{[2]_{q}-1}\left( 1-\dfrac{[3]_{q}-1}{[2]_{q}-1}\mu \right)|a_{2}|^{2} \right] \leq \dfrac{b_{1}}{[3]_{q}-1}.\\
\end{equation}

\end{cor}
\vspace{5mm}
Next, we obtain the coefficient bound for the functions belonging to the class  $\mathcal{C}_{p,q}(\phi)$ :\\

\begin{thm}
	Let $\phi (z)=1+b_{1}z+b_{2}z^{2} \dots $ with $b_{1}> 0$ and $b_{2}\geq 0$. Let
	\begin{align}
	\rho _{1}&=\dfrac{[2]_{p,q}^{2}([2]_{p,q}-1)b^{2}_{1}+([2]_{p,q}[2]_{p,q}-1)^{2}(b_{2}-b_{1})}{[3]_{p,q}([3]_{p,q}-1)b^{2}_{1}},\label{66}\\
	\rho _{2}&=\dfrac{[2]_{p,q}^{2}([2]_{p,q}-1)b^{2}_{1}+([2]_{p,q}[2]_{p,q}-1)^{2}(b_{2}+b_{1})}{[3]_{p,q}([3]_{p,q}-1)b^{2}_{1}},\label{67}\\
	\rho _{3}&=\dfrac{[2]_{p,q}^{2}([2]_{p,q}-1)b^{2}_{1}+([2]_{p,q}[2]_{p,q}-1)^{2} b_{2}}{[3]_{p,q}([3]_{p,q}-1)b^{2}_{1}}.\label{68}
	\end{align}
	If $f$, given by equation (\ref{1}), belongs to the class $\mathcal{C}_{p,q}(\phi)$, then
	\begin{eqnarray}\label{69}
	|a_{3}-\mu a^{2}_{2}|\leq \left\{
	\begin{array}{lll}
	\dfrac{b_{2}}{[3]_{p,q}([3]_{p,q}-1)}+\dfrac{b^{2}_{1}}{[2]_{p,q}-1}\left(\dfrac{1}{[3]_{p,q}([3]_{p,q}-1)}-\dfrac{\mu}{[2]^{2}_{p,q}([2]_{p,q}-1)}\right)  , &if ~~ \mu \leq \rho_{1};
	\\
	\\
	\dfrac{b_{1}}{[3]_{p,q}([3]_{p,q}-1)}, & if ~~ \rho_{1}\leq \mu \leq \rho_{2};
	\\
	\\
	\dfrac{-b_{2}}{[3]_{p,q}([3]_{p,q}-1)}-\dfrac{b^{2}_{1}}{[2]_{p,q}-1}\left(\dfrac{1}{[3]_{p,q}([3]_{p,q}-1)}-\dfrac{\mu}{[2]^{2}_{p,q}([2]_{p,q}-1)}\right ) , & if~~ \mu \geq \rho_{2}.
	\end{array}\right.
	\end{eqnarray}\\ 
	
	Further, if $\rho_{1}< \mu \leq \rho_{3}$, then 
	\begin{equation}\label{70}
	|a_{3}-\mu a^{2}_{2}|+\dfrac{[2]^{2}_{p,q}([2]_{p,q}-1)^{2}}{[3]_{p,q}([3]_{p,q}-1)b^{2}_{1}}\left[ b_{1}-b_{2}-\dfrac{b^{2}_{1}}{[2]_{p,q}-1}\left( 1-\dfrac{[3]_{p,q}([3]_{p,q}-1)}{[2]^{2}_{p,q}([2]_{p,q}-1)}\mu \right)|a_{2}|^{2} \right] \leq \dfrac{b_{1}}{[3]_{p,q}([3]_{p,q}-1)}\\
	\end{equation}
	and if 
	$\rho_{3}\leq \mu <\rho_{2}$, then 
	\begin{equation}\label{71}
	|a_{3}-\mu a^{2}_{2}|+\dfrac{[2]^{2}_{p,q}([2]_{p,q}-1)^{2}}{[3]_{p,q}([3]_{p,q}-1)b^{2}_{1}}\left[ b_{1}+b_{2}+\dfrac{b^{2}_{1}}{[2]_{p,q}-1}\left( 1-\dfrac{[3]_{p,q}([3]_{p,q}-1)}{[2]^{2}_{p,q}([2]_{p,q}-1)}\mu \right)|a_{2}|^{2} \right] \leq \dfrac{b_{1}}{[3]_{p,q}([3]_{p,q}-1)}.\\
	\end{equation}

\end{thm}
\begin{proof}[\textbf{Proof}]
For $v\leq 0$, equation (\ref{41}) gives
	\begin{equation*}
	\mu \leq \dfrac{[2]_{p,q}^{2}([2]_{p,q}-1)b^{2}_{1}+([2]_{p,q}[2]_{p,q}-1)^{2}(b_{2}-b_{1})}{[3]_{p,q}([3]_{p,q}-1)b^{2}_{1}}.
	\end{equation*}
	Let  $\dfrac{[2]_{p,q}^{2}([2]_{p,q}-1)b^{2}_{1}+([2]_{p,q}[2]_{p,q}-1)^{2}(b_{2}-b_{1})}{[3]_{p,q}([3]_{p,q}-1)b^{2}_{1}}=\rho_{1}$, then from the above relation we have $\mu \leq \rho_{1}$.\\
	
	Let $p(z)$ be a function given by equation (\ref{111}) with $\Re\left( p(z)\right)>0 $ and $f(z)$, given by equation (\ref{1}), be a member of the class $\mathcal{C}_{p,q}(\phi)$, then equation (\ref{42}) holds. Thus, using Lemma 3.1,  for $v\leq 0$, in  equation (\ref{42}), we get
	\begin{equation*}
	|a_{3}-\mu a^{2}_{2}|\leq \dfrac{b_{1}}{2 [3]_{p,q}([3]_{p,q}-1)}(-4v+2),
	\end{equation*}
	which on using equation (\ref{41}), gives
	\begin{equation}\label{72}
	|a_{3}-\mu a^{2}_{2}|\leq \dfrac{b_{1}}{[3]_{p,q} ([3]_{p,q}-1)}\left(\dfrac{b_{2}}{b_{1}}+\dfrac{b_{1}}{[2]_{p,q}-1}\left(1-\dfrac{[3]_{p,q}([3]_{p,q}-1)}{[2]^{2}_{p,q}([2]_{p,q}-1)} \mu \right) \right),
	\end{equation}
	where $\mu \leq \rho_{1}$.\\

Inequality  (\ref{72}) gives the first inequality of assertion (\ref{69}).\\
	
Again, if we take $0\leq v\leq 1$, then equation (\ref{41}) gives
	\begin{equation*}
	\rho_{1} \leq \mu \leq \dfrac{[2]_{p,q}^{2}([2]_{p,q}-1)b^{2}_{1}+([2]_{p,q}[2]_{p,q}-1)^{2}(b_{2}+b_{1})}{[3]_{p,q}([3]_{p,q}-1)b^{2}_{1}}.\\
	\end{equation*}
	
	Let $\dfrac{[2]_{p,q}^{2}([2]_{p,q}-1)b^{2}_{1}+([2]_{p,q}[2]_{p,q}-1)^{2}(b_{2}+b_{1})}{[3]_{p,q}([3]_{p,q}-1)b^{2}_{1}}=\rho _{2}$, then $  \rho_{1} \leq\mu \leq \rho_{2}$, 	where $\rho_{1}$ is given by equation (\ref{66}).\\
	
Now, using  Lemma 3.1, for $0\leq v \leq 1$,  in equation (\ref{42}), we get
	\begin{equation*}
	|a_{3}-\mu a^{2}_{2}|\leq \dfrac{b_{1}}{ [3]_{p,q}([3]_{p,q}-1)},
	\end{equation*}
	which gives the second inequality of assertion (\ref{69}).\\

Next, if we take $ v\geq 1$, then equation (\ref{41}) gives that $\mu\geq \rho_{2}$.\\

Now, using Lemma 3.1, for $v\geq 1$ in equation (\ref{42}), we get
\begin{equation*}
	|a_{3}-\mu a^{2}_{2}|\leq \dfrac{|b_{1}|}{2 [3]_{p,q}([3]_{p,q}-1)}(4v-2),
\end{equation*}
which on using equation (\ref{41}) gives
	\begin{equation}\label{73}
	|a_{3}-\mu a^{2}_{2}| \leq \dfrac{b_{1}}{ [3]_{p,q}([3]_{p,q}-1)}\left(-\dfrac{b_{2}}{b_{1}}-\dfrac{b_{1}}{[2]_{p,q}-1}\left(1-\dfrac{[3]_{p,q}([3]_{p,q}-1)}{[2]^{2}_{p,q}([2]_{p,q}-1)} \mu \right) \right),
	\end{equation}
	where $\mu \geq \rho_{2}$.\\
	
Simplifying the right hand side of inequality  (\ref{73}), we get the third  inequality of assertion (\ref{69}).\\

	Further, if $ 0 <v\leq \dfrac{1}{2}$, then using equation (\ref{41}), we have
	\begin{equation*}
	0	<\dfrac{1}{2}\left( 1-\dfrac{b_{2}}{b_{1}}-\dfrac{b_{1}}{[2]_{p,q}-1}\left(1-\dfrac{[3]_{p,q}([3]_{p,q}-1)}{[2]^{2}_{p,q}([2]_{p,q}-1)} \mu \right) \right)\leq \dfrac{1}{2},
	\end{equation*}
	which on simplifying, gives 
	\begin{equation}\label{102}
	\rho_{1}	<	\mu \leq 	\dfrac{[2]_{p,q}^{2}([2]_{p,q}-1)b^{2}_{1}+([2]_{p,q}[2]_{p,q}-1)^{2} b_{2}}{[3]_{p,q}([3]_{p,q}-1)b^{2}_{1}}.\\
	\end{equation}
		
	Let
$\dfrac{[2]_{p,q}^{2}([2]_{p,q}-1)b^{2}_{1}+([2]_{p,q}[2]_{p,q}-1)^{2} b_{2}}{[3]_{p,q}([3]_{p,q}-1)b^{2}_{1}}=\rho _{3}$, then from (\ref{102}), we have $ \rho_{1} < \mu \leq \rho_{3}$, where $\rho _{1}$ is given by (\ref{66}).\\
	
	Now, using equations (\ref{38}) and (\ref{66}), we get

	\begin{equation}\label {75}
	\begin{split}
		&|a_{3}-\mu a^{2}_{2}|+(\mu - \rho_{1})|a_{2}|^{2}=|a_{3}-\mu a^{2}_{2}|\\
	&\hspace{30mm}	+\left( \mu - \dfrac{[2]_{p,q}^{2}([2]_{p,q}-1)b^{2}_{1}+([2]_{p,q}[2]_{p,q}-1)^{2}(b_{2}-b_{1})}{[3]_{p,q}([3]_{p,q}-1)b^{2}_{1}}\right)  \dfrac{b^{2}_{1}|c_{1}|^{2}}{4[2]^{2}_{p,q}([2]_{p,q}-1)^{2}}, \\
	\end{split}
	\end{equation}
which on using equation (\ref{42}), we get
	\begin{equation}\label{76}
		\begin{split}
	&	|a_{3}-\mu a^{2}_{2}|+(\mu - \rho_{1})|a_{2}|^{2}=\\
	& \hspace{15mm}\dfrac{b_{1}}{ 2 [3]_{p,q}([3]_{p,q}-1)} \left(  |c_{2}-vc^{2}_{1}|+\dfrac{1}{2}\left( 1-\dfrac{b_{2}}{b_{1}}-\dfrac{b_{1}}{[2]_{p,q}-1}\left(1-\dfrac{[3]_{p,q}([3]_{p,q}-1)}{[2]^{2}_{p,q}([2]_{p,q}-1)} \mu \right) \right)|c_{1}|^{2}  \right),\\
		\end{split}
	\end{equation}

Again, using equation (\ref{41}) in equation (\ref{76}), we have
\begin{equation*}
|a_{3}-\mu a^{2}_{2}|+(\mu - \rho_{1})|a_{2}|^{2}=\dfrac{b_{1}}{  [3]_{p,q}([3]_{p,q}-1)}\left( \dfrac{1}{2}(|c_{2}-vc^{2}_{1}|+v|c_{1}|^{2})\right),
\end{equation*}
which in view of inequality (\ref{100}) gives
	\begin{equation}\label{77}
	|a_{3}-\mu a^{2}_{2}|+(\mu - \rho_{1})|a_{2}|^{2} \leq \dfrac{b_{1}}{ [3]_{p,q}([3]_{p,q}-1)}.\\
	\end{equation}
	
Now, using equation (\ref{38}) and  inequality (\ref{77}) in equation (\ref{75}), we get
	\begin{equation*}
	|a_{3}-\mu a^{2}_{2}|+\left( \mu - \dfrac{[2]_{p,q}^{2}([2]_{p,q}-1)b^{2}_{1}+([2]_{p,q}[2]_{p,q}-1)^{2}(b_{2}-b_{1})}{[3]_{p,q}([3]_{p,q}-1)b^{2}_{1}}\right)  |a_{2}|^{2}\leq \dfrac{b_{1}}{[3]_{p,q} ([3]_{p,q}-1)}.\\
	\end{equation*}
	
Simplifying the above inequality, we obtain the assertion (\ref{70}).\\
	
	Similarly, if $\dfrac{1}{2} \leq v <1$, then using equation (\ref{41}), we get   $ \rho_{3} \leq \mu < \rho_{2}$.\\
	
	Now, using equations (\ref{38}) and  (\ref{67}), we get
	\begin{equation}\label {82}
\begin{split}
	&|a_{3}-\mu a^{2}_{2}|+( \rho_{2}-\mu)|a_{2}|^{2}=\\
	&\hspace{15mm}|a_{3}-\mu a^{2}_{2}|+\left( \dfrac{[2]_{p,q}^{2}([2]_{p,q}-1)b^{2}_{1}+([2]_{p,q}[2]_{p,q}-1)^{2}(b_{2}+b_{1})}{[3]_{p,q}([3]_{p,q}-1)b^{2}_{1}}-\mu  \right) \dfrac{b^{2}_{1}|c_{1}|^{2}}{4[2]^{2}_{p,q}([2]_{p,q}-1)^{2}},\\
\end{split}
	\end{equation}
Using equation (\ref{42}) in equation  (\ref{82}) and then simplifying, we get
\begin{equation*}
\begin{split}
&|a_{3}-\mu a^{2}_{2}|+( \rho_{2}-\mu)|a_{2}|^{2}=\\
&\hspace{15mm}\dfrac{b_{1}}{ 2[3]_{p,q}([3]_{p,q}-1)}\left( |c_{2}-vc^{2}_{1}|+\dfrac{1}{2}\left(1+\dfrac{b_{2}}{b_{1}}+\dfrac{b_{1}}{[2]_{p,q}-1}\left( 1-\dfrac{[3]_{p,q}([3]_{p,q}-1)}{[2]_{p,q}^{2}([2]_{p,q}-1)}\mu\right)|c_{1}|^{2} \right) \right)  ,\\
\end{split}
\end{equation*}
which on using equation (\ref{41}), gives
	\begin{equation}\label{84}
	|a_{3}-\mu a^{2}_{2}|+( \rho_{2}-\mu)|a_{2}|^{2}=\dfrac{|b_{1}|}{ [3]_{p,q}([3]_{p,q}-1)}\left( \dfrac{1}{2}\left(  |c_{2}-vc^{2}_{1}|+(1-v)|c_{1}|^{2}  \right) \right).\\
	\end{equation}
	
	Now, since  $\dfrac{1}{2} \leq v <1$, therefore using inequality (\ref{101}) of Lemma 3.1 in equation (\ref{84}), we get
	\begin{equation}\label{85}
	|a_{3}-\mu a^{2}_{2}|+( \rho_{2}-\mu)|a_{2}|^{2} \leq \dfrac{b_{1}}{[3]_{p,q} ([3]_{p,q}-1)}.\\
	\end{equation}
	
Using inequality (\ref{85}) in equation (\ref{82}), gives
	\begin{equation*}
	|a_{3}-\mu a^{2}_{2}|+\left( \dfrac{[2]_{p,q}^{2}([2]_{p,q}-1)b^{2}_{1}+([2]_{p,q}[2]_{p,q}-1)^{2}(b_{2}+b_{1})}{[3]_{p,q}([3]_{p,q}-1)b^{2}_{1}}-\mu\right)  |a_{2}|^{2}\leq \dfrac{b_{1}}{[3]_{p,q} ([3]_{p,q}-1)},
	\end{equation*}
	where $\rho_{3}\leq \mu < \rho_{2}$.\\ 
	
	Finally, on simplifying the above inequality, we obtain assertion  (\ref{71}).\\	
\end{proof}

For $p=1$, Theorem 2.2, gives the following corollary for the class $\mathcal{C}_{q}(\phi)$:
\begin{cor}
		Let $\phi (z)=1+b_{1}z+b_{2}z^{2} \dots $ with $b_{1}> 0$ and $b_{2}\geq 0$. Let
	\begin{align}
	\rho _{1}&=\dfrac{[2]_{q}^{2}([2]_{q}-1)b^{2}_{1}+([2]_{q}[2]_{q}-1)^{2}(b_{2}-b_{1})}{[3]_{q}([3]_{q}-1)b^{2}_{1}},\\
	\rho _{2}&=\dfrac{[2]_{q}^{2}([2]_{q}-1)b^{2}_{1}+([2]_{q}[2]_{q}-1)^{2}(b_{2}+b_{1})}{[3]_{q}([3]_{q}-1)b^{2}_{1}},\\
	\rho _{3}&=\dfrac{[2]_{q}^{2}([2]_{q}-1)b^{2}_{1}+([2]_{q}[2]_{q}-1)^{2} b_{2}}{[3]_{p,q}([3]_{q}-1)b^{2}_{1}}.
	\end{align}
	If $f$, given by equation (\ref{1}), belongs to the class $\mathcal{C}_{q}(\phi)$, then
	\begin{eqnarray}
	|a_{3}-\mu a^{2}_{2}|\leq \left\{
	\begin{array}{lll}
	\dfrac{b_{2}}{[3]_{q}([3]_{q}-1)}+\dfrac{b^{2}_{1}}{[2]_{q}-1}\left(\dfrac{1}{[3]_{q}([3]_{q}-1)}-\dfrac{\mu}{[2]^{2}_{q}([2]_{q}-1)}\right)  , &if ~~ \mu \leq \rho_{1};
	\\
	\\
	\dfrac{b_{1}}{[3]_{q}([3]_{q}-1)}, & if ~~ \rho_{1}\leq \mu \leq \rho_{2};
	\\
	\\
	\dfrac{-b_{2}}{[3]_{q}([3]_{q}-1)}-\dfrac{b^{2}_{1}}{[2]_{q}-1}\left(\dfrac{1}{[3]_{q}([3]_{q}-1)}-\dfrac{\mu}{[2]^{2}_{q}([2]_{q}-1)}\right ) , & if~~ \mu \geq \rho_{2}.
	\end{array}\right.
	\end{eqnarray}\\ 
	
	Further, if $\rho_{1}<\mu \leq \rho_{3}$, then 
	\begin{equation}
	|a_{3}-\mu a^{2}_{2}|+\dfrac{[2]^{2}_{q}([2]_{q}-1)^{2}}{[3]_{q}([3]_{q}-1)b^{2}_{1}}\left[ b_{1}-b_{2}-\dfrac{b^{2}_{1}}{[2]_{q}-1}\left( 1-\dfrac{[3]_{q}([3]_{q}-1)}{[2]^{2}_{q}([2]_{q}-1)}\mu \right)|a_{2}|^{2} \right] \leq \dfrac{b_{1}}{[3]_{q}([3]_{q}-1)}.\\
	\end{equation}
	and if 
	$\rho_{3}\leq \mu < \rho_{2}$, then 
	\begin{equation}
	|a_{3}-\mu a^{2}_{2}|+\dfrac{[2]^{2}_{q}([2]_{q}-1)^{2}}{[3]_{q}([3]_{q}-1)b^{2}_{1}}\left[ b_{1}+b_{2}+\dfrac{b^{2}_{1}}{[2]_{q}-1}\left( 1-\dfrac{[3]_{q}([3]_{q}-1)}{[2]^{2}_{q}([2]_{q}-1)}\mu \right)|a_{2}|^{2} \right] \leq \dfrac{b_{1}}{[3]_{q}([3]_{q}-1)}.\\
	\end{equation}
	
\end{cor}
\vspace{5mm}
In the next section, we discuss some applications of the results, established in Sections 1. and 2. .
\section{Application}

\hspace{5mm}We recall that the Bernardi integral operator $\mathcal{F}_{c}$ is given by \cite{Ber}:
\begin{equation*}
\mathcal{F}_{c}(f(z))=\dfrac{1+c}{z^{c}} \int_{0}^{z}t^{c-1} f(t)dt \hspace{5mm}(f\in \mathcal{A}, \  c>-1).\\
\end{equation*}

Now, in view of above equation, we introduce the  $(p,q)$-Bernardi integral operator $\mathcal{L}(z)$ as:

\begin{equation}\label{90}
\mathcal{L}(z):=\mathcal{F}_{c,p,q}(f(z))=\dfrac{[1+c]_{p,q}}{z^{\beta}} \int_{0}^{z}t^{c-1} f(t)d_{p,q}t \hspace{5mm} c=0,1,2,3,\dots.\\
\end{equation}

Let $f\in \mathcal{A}$, then using equations (\ref{4}) and (\ref{105}), we obtain the following power series for the  function  $\mathcal{L}$ in the open unit disc $ \mathbb{U}=\left\lbrace z\in \mathbb{C}:|z|<1\right\rbrace $:
	\begin{equation}\label{91}
\mathcal{L}(z)=z+\sum_{n=2}^{\infty}\dfrac{[1+c]_{p,q}}{[n+c]_{p,q}}a_{n}z^{n} \hspace{5mm} (c=1,2,3,\dots; \  0<q<p\leq1; \  f\in \mathcal{A}).\\
\end{equation}

	It is clear that  $\mathcal{L}(z)$ is analytic in open disc $\mathbb{U}$.\\
	
We note that, by taking $p=1$ in equation (\ref{90}), we get $q$-Bernardi integral operator  \cite{noor2017q} .\\

Let
\begin{equation}\label{93}
L_{n}=\dfrac{[1+c]_{p,q}}{[n+c]_{p,q}}, \hspace{10mm} n\geq 1.\\
\end{equation}

Now, applying the Theorem 2.1 to the function $\mathcal{L}(z)$, defined by equation (\ref{91}), we get the following application of the theorem:\\

\textbf{\rom{1}.}	Let $\phi (z)=1+b_{1}z+b_{2}z^{2} \dots $, with $b_{1}\neq 0$. If $\mathcal{L}$, given by equation (\ref{91}), belongs to the class $\mathcal{S}^{*}_{p,q}(\phi)$, then 
\begin{equation*}
|a_{3}-\mu a^{2}_{2}|\leq \dfrac{|b_{1}|}{[3]_{p,q}L_{3}-1} \max \left\lbrace 1; \left|\dfrac{ b_{2}}{b_{1}}+\dfrac{b_{1}}{[2]_{p,q}L_{2}-1}\left(1-\dfrac{[3]_{p,q}L_{3}-1}{[2]_{p,q}L_{2}-1} \right) \mu \right|  \right\rbrace ,
\end{equation*}
where $L_{2}$ and $L_{3}$ are given by equation (\ref{93}), $b_{1}, b_{2}, \dots \in \mathbb{R} $, $\mu \in \mathbb{C}$, $0<q<p\leq 1$.\\

Next, applying the Theorem 2.2 to the function $\mathcal{L}(z)$, defined by equation (\ref{91}), we get the following application of the theorem:\\

\textbf{\rom{2}.}	Let $\phi (z)=1+b_{1}z+b_{2}z^{2} \dots $, with $b_{1}\neq 0$. If $\mathcal{L}$, given by equation (\ref{91}), belongs to the class $\mathcal{C}_{p,q}(\phi)$, then 
	\begin{equation*}
	|a_{3}-\mu a^{2}_{2}|\leq \dfrac{|b_{1}|}{[3]_{p,q}L_{3}([3]_{p,q}L_{3}-1)} \max \left\lbrace 1; \left|\dfrac{ b_{2}}{b_{1}}+\dfrac{b_{1}}{[2]_{p,q}-1}L_{2}\left(1-\dfrac{[3]_{p,q}L_{3}([3]_{p,q}L_{3}-1)}{[2]^{2}_{p,q}L_{2}([2]_{p,q}L_{2}-1)} \right) \mu \right|  \right\rbrace ,
	\end{equation*}
where $L_{2}$ and $L_{3}$ are given by equation (\ref{93}), $b_{1}, b_{2}, \dots \in \mathbb{R} $, $\mu \in \mathbb{C}$, $0<q<p\leq 1$.\\

Further, applying the Theorem 3.1 to the function $\mathcal{L}(z)$, defined by equation (\ref{91}), we get the following application of the theorem: \\

	\textbf{\rom{3}.}	Let $\phi (z)=1+b_{1}z+b_{2}z^{2} \dots $  with $b_{1}> 0$ and $b_{2}\geq 0$. Let
	\begin{align}
	\sigma _{1}&=\dfrac{([2]_{p,q}L_{2}-1)b^{2}_{1}+([2]_{p,q}L_{2}-1)^{2}(b_{2}-b_{1})}{([3]_{p,q}L_{3}-1)b^{2}_{1}},\nonumber\\
	\sigma _{2}&=\dfrac{([2]_{p,q}L_{2}-1)b^{2}_{1}+([2]_{p,q}L_{2}-1)^{2}(b_{2}+b_{1})}{([3]_{p,q}L_{3}-1)b^{2}_{1}},\nonumber\\
	\sigma _{3}&=\dfrac{([2]_{p,q}L_{2}-1)b^{2}_{1}+([2]_{p,q}L_{2}-1)^{2} b_{2}}{([3]_{p,q}L_{3}-1)b^{2}_{1}}.\nonumber
	\end{align}
	If $\mathcal{L}$, given by equation (\ref{91}), belongs to the class $\mathcal{S}^{*}_{p,q}(\phi)$, then
	\begin{eqnarray}
	|a_{3}-\mu a^{2}_{2}|\leq \left\{
	\begin{array}{lll}
	\dfrac{b_{2}}{[3]_{p,q}-1}+\dfrac{b^{2}_{1}}{[2]_{p,q}-1}\left(\dfrac{1}{[3]_{p,q}-1}-\dfrac{\mu}{[2]_{p,q}-1}\right)  , &if ~~ \mu \leq \sigma_{1};
	\\
	\\
	\dfrac{b_{1}}{[3]_{p,q}-1}, & if ~~ \sigma_{1}\leq \mu \leq \sigma_{2};
	\\
	\\
	\dfrac{-b_{2}}{[3]_{p,q}-1}-\dfrac{b^{2}_{1}}{[2]_{p,q}-1}\left(\dfrac{1}{[3]_{p,q}-1}-\dfrac{\mu}{[2]_{p,q}-1}\right) , & if~~ \mu \geq \sigma_{2}.
	\end{array}\right.
	\end{eqnarray}\\

	Further, if $\sigma_{1}< \mu \leq \sigma_{3}$, then 
	\begin{equation*}
	|a_{3}-\mu a^{2}_{2}|+\dfrac{([2]_{p,q}L_{2}-1)^{2}}{([3]_{p,q}L_{3}-1)b^{2}_{1}}\left[ b_{1}-b_{2}-\dfrac{b^{2}_{1}}{[2]_{p,q}L_{2}-1}\left( 1-\dfrac{[3]_{p,q}L_{3}-1}{[2]_{p,q}L_{2}-1}\mu \right)|a_{2}|^{2} \right] \leq \dfrac{b_{1}}{[3]_{p,q}L_{3}-1}.\\
	\end{equation*}
	and if 
	$\sigma_{3}\leq \mu <\sigma_{2}$, then 
	\begin{equation*}
	|a_{3}-\mu a^{2}_{2}|+\dfrac{([2]_{p,q}L_{2}-1)^{2}}{([3]_{p,q}L_{3}-1)b^{2}_{1}}\left[ b_{1}+b_{2}+\dfrac{b^{2}_{1}}{[2]_{p,q}L_{2}-1}\left( 1-\dfrac{[3]_{p,q}L_{3}-1}{[2]_{p,q}L_{2}-1}\mu \right)|a_{2}|^{2} \right] \leq \dfrac{b_{1}}{[3]_{p,q}L_{3}-1},\\
	\end{equation*}
	where $L_{2}$ and $L_{3}$ are given by equation (\ref{93}).\\

Finally, applying the Theorem 3.1 to the function $\mathcal{L}(z)$, defined by equation (\ref{91}), we get the following application of the theorem:\\

	\textbf{\rom{4}.}	Let $\phi (z)=1+b_{1}z+b_{2}z^{2} \dots $ with $b_{1}> 0$ and $b_{2}\geq 0$. Let
	\begin{align}
	\rho_{1}&=\dfrac{[2]_{p,q}^{2}([2]_{p,q}L_{2}-1)b^{2}_{1}+([2]_{p,q}L_{2}[2]_{p,q}L_{2}-1)^{2}(b_{2}-b_{1})}{[3]_{p,q}L_{3}([3]_{p,q}L_{3}-1)b^{2}_{1}},\nonumber\\
	\rho_{2}&=\dfrac{[2]_{p,q}^{2}L_{2}([2]_{p,q}L_{2}-1)b^{2}_{1}+[2]^{2}_{p,q}L_{2}([2]_{p,q}L_{2}-1)^{2}(b_{2}+b_{1})}{[3]_{p,q}L_{3}([3]_{p,q}L_{3}-1)b^{2}_{1}},\nonumber\\
	\rho_{3}&=\dfrac{[2]_{p,q}^{2}L_{2}([2]_{p,q}L_{2}-1)b^{2}_{1}+[2]^{2}_{p,q}L_{2}([2]_{p,q}L_{2}-1)^{2} b_{2}}{[3]_{p,q}L_{3}([3]_{p,q}L_{3}-1)b^{2}_{1}}.\nonumber
	\end{align}
	If $\mathcal{L}$, given by equation (\ref{91}), belongs to the class $\mathcal{C}_{p,q}(\phi)$, then
	\begin{eqnarray}
	|a_{3}-\mu a^{2}_{2}|\leq \left\{
	\begin{array}{lll}
	\dfrac{b_{2}}{[3]_{p,q}L_{3}([3]_{p,q}L_{3}-1)}+\dfrac{b^{2}_{1}}{[2]_{p,q}L_{2}-1}\left(\dfrac{1}{[3]_{p,q}L_{3}([3]_{p,q}L_{3}-1)}-\dfrac{\mu}{[2]^{2}_{p,q}L_{2}([2]_{p,q}L_{2}-1)}\right)  , &if ~~ \mu \leq \rho_{1};
	\\
	\\
	\dfrac{b_{1}}{[3]_{p,q}L_{3}([3]_{p,q}L_{3}-1)}, & if ~~ \rho_{1}\leq \mu \leq \rho_{2};
	\\
	\\
	\dfrac{-b_{2}}{[3]_{p,q}L_{3}([3]_{p,q}L_{3}-1)}-\dfrac{b^{2}_{1}}{[2]_{p,q}L_{2}-1}\left(\dfrac{1}{[3]_{p,q}L_{3}([3]_{p,q}L_{3}-1)}-\dfrac{\mu}{[2]^{2}_{p,q}L_{2}([2]_{p,q}L_{2}-1)}\right ) , & if~~ \mu \geq \rho_{2}.
	\end{array}\right.
	\end{eqnarray}\\ 
	
	Further, if $\rho_{1}< \mu \leq \rho_{3}$, then 
	\begin{equation*}
	|a_{3}-\mu a^{2}_{2}|+\dfrac{[2]^{2}_{p,q}L_{2}([2]_{p,q}L_{2}-1)^{2}}{[3]_{p,q}L_{3}([3]_{p,q}L_{3}-1)b^{2}_{1}}\left[ b_{1}-b_{2}-\dfrac{b^{2}_{1}}{[2]_{p,q}L_{2}-1}\left( 1-\dfrac{[3]_{p,q}L_{3}([3]_{p,q}L_{3}-1)}{[2]^{2}_{p,q}L_{2}([2]_{p,q}L_{2}-1)}\mu \right)|a_{2}|^{2} \right] \leq \dfrac{b_{1}}{[3]_{p,q}L_{3}([3]_{p,q}L_{3}-1)}.\\
	\end{equation*}
	and if 
	$\rho_{3}\leq \mu <\rho_{2}$, then 
	\begin{equation*}
	|a_{3}-\mu a^{2}_{2}|+\dfrac{[2]^{2}_{p,q}L_{2}([2]_{p,q}L_{2}-1)^{2}}{[3]_{p,q}L_{3}([3]_{p,q}L_{3}-1)b^{2}_{1}}\left[ b_{1}+b_{2}+\dfrac{b^{2}_{1}}{[2]_{p,q}L_{2}-1}\left( 1-\dfrac{[3]_{p,q}L_{3}([3]_{p,q}L_{3}-1)}{[2]^{2}_{p,q}L_{2}([2]_{p,q}L_{2}-1)}\mu \right)|a_{2}|^{2} \right] \leq \dfrac{b_{1}}{[3]_{p,q}L_{3}([3]_{p,q}L_{3}-1)},
	\end{equation*}
	where $L_{2}$ and $L_{3}$ are given by equation (\ref{93}).\\
	
	\section{Conclusion}
	
	In our results, by using the $(p,q)$-derivative operator, the generalized classes of  $(p,q)$-starlike and $(p,q)$-convex functions were introduced which are a generalization of starlike and convex functions. Also, the $(p,q)$-Bernardi integral operator for analytic functions were defined in the open unit disc $\mathbb{U}=\left\lbrace z\in \mathbb{C}:|z|<1\right\rbrace $ . In our main results, the Fekete-Szeg{\"o} inequalities. For the validity of our results can be applicable for our introduced $(p,q)$-Bernardi integral operator. Moreover, Some special cases of the results were  established . Further, certain applications of the main results for the $(p,q)$-starlike and $(p,q)$-convex functions were obtained by applying the  $(p,q)$-Bernardi integral operator.

 \end{document}